\documentclass[11pt]{article}
\usepackage{amsthm,amsmath,amssymb,graphicx,color,float}
\usepackage[margin=25mm]{geometry}
\usepackage[font=small]{caption}
\usepackage{hyperref}

\newcommand{\floor}[1]{\lfloor{#1}\rfloor}
\newcommand{\ceil}[1]{\lceil{#1}\rceil}

\DeclareMathOperator{\td}{td}
\DeclareMathOperator{\tw}{tw}
\DeclareMathOperator{\pw}{pw}
\DeclareMathOperator{\rpw}{rpw}

\newtheorem{theorem}{Theorem}
\newtheorem{lemma}[theorem]{Lemma}
\newtheorem{proposition}[theorem]{Proposition}

\newtheorem{conjecture}[theorem]{Conjecture}

\begin{document}

\title{\bf Circumference and Pathwidth \\ of Highly Connected Graphs}

\author{Emily A. Marshall\footnote{Department of Mathematics, Vanderbilt University, Nashville, Tennessee, U.S.A. (\texttt{emily.a.marshall@vanderbilt.edu}). Research supported by the National Science Foundation under Grant No. 1310758} \qquad David R. Wood\footnote{School of Mathematical Sciences, Monash University, Melbourne, Australia (\texttt{david.wood@monash.edu}). Research supported by the Australian Research Council. }}

\maketitle

\begin{abstract}
Birmele [\emph{J.\ Graph Theory}, 2003] proved that every graph with circumference $t$ has treewidth at most $t-1$. Under the additional assumption of 2-connectivity, such graphs have bounded pathwidth, which is a qualitatively stronger conclusion.  Birmele's theorem was extended by Birmele, Bondy and Reed [\emph{Combinatorica}, 2007]  who showed that every graph without $k$ disjoint cycles of length at least $t$ has treewidth $\mathcal{O}(tk^2)$. Our main result states that, under the additional  assumption of $(k+1)$-connectivity, such graphs have bounded pathwidth. In fact, they have pathwidth $\mathcal{O}(t^3+tk^2)$.  Moreover, examples show that $(k+1)$-connectivity is required for bounded pathwidth to hold.  These results suggest the following general question: for which values of $k$ and graphs $H$ does every $k$-connected $H$-minor-free graph have bounded pathwidth? We discuss this question and provide a few observations. 
\end{abstract}

\section{Introduction}

Birmele~\cite{birmele} proved that every graph with circumference $t$ has treewidth at most $t-1$, and this bound is tight for the complete graph $K_t$. Ne\v{s}et\v{r}il and Ossona de Mendez \cite[page 118]{nesetril} showed that under the additional assumption of 2-connectivity, such graphs have treedepth at most $1+(t-2)^2$.  Since pathwidth is at most treedepth minus 1, every 2-connected graph with circumference $t$ has pathwidth at most $(t-2)^2$.  Our first result strengthens this bound.

\begin{theorem}
\label{thm:2-conn}
Every 2-connected graph with circumference $t$ has pathwidth at most $\floor{\tfrac{t}{2}}(t-1)$.
\end{theorem}

The 2-connectivity assumption is needed in Theorem~\ref{thm:2-conn} since complete binary trees have unbounded pathwidth. In particular, the complete binary tree of height $h$ has pathwidth $\ceil{\frac{h}{2}}$.

Birmele's theorem was extended by Birmele, Bondy and Reed~\cite{bondy}, who showed that graphs without $k$ disjoint cycles of length at least $t$ have treewidth $\mathcal{O}(tk^2)$. Under the additional assumption of $(k+1)$-connectivity, we prove that such graphs have bounded pathwidth. 

\begin{theorem}
\label{thm:disjointcycles} 
Every $(k+1)$-connected graph without $k$ disjoint cycles of length at least $t$ has pathwidth at most $\mathcal{O}( t^3+tk^2)$. 
\end{theorem}

We now show that the assumption of $(k+1)$-connectivity is needed in Theorem~\ref{thm:disjointcycles}. Suppose on the contrary that every $k$-connected graph without $k$ disjoint cycles of length at least $t$ has pathwidth at most $f(k,t)$ for some function $f$. Let $G$ be the graph obtained from the complete binary tree of height $h$ by adding $k-1$ dominant vertices. Observe that $G$ is $k$-connected. Since every cycle in $G$ uses at least one of the dominant vertices, $G$ contains no $k$ disjoint cycles. Thus $G$ has pathwidth at most $f(k,t)$ for all $t\geq 3$. On the other hand, the pathwidth of $G$ equals $\ceil{\frac{h}{2}}+k-1$. We obtain a contradiction by choosing $h>2\cdot f(k,t)$.

The proofs of Theorems~\ref{thm:2-conn} and \ref{thm:disjointcycles} are given in Sections~\ref{sec:2-conn} and \ref{proofdisjointcycles} respectively. We conclude in Section~\ref{minors} by re-interpreting these results in terms of excluded minors. In general, we observe that highly connected $H$-minor-free graphs have bounded pathwidth. Determining the minimum connectivity required for this behaviour to occur is an interesting line of future research.

\section{Definitions}

Let $G$ be an (undirected, simple, finite) graph. The \emph{circumference} of $G$ is the length of the longest cycle in $G$, or is 0 if $G$ is  acyclic. 
A \emph{tree decomposition} $(T,\{B_x \subseteq V(G): x\in V(T)\})$ of $G$ consists of a tree $T$ and a set $\{B_x\subseteq V(G): x \in V(T)\}$ of sets of vertices of $G$ indexed by the nodes of $T$, such that:
		\begin{itemize}
				\item for each vertex $v\in V(G)$, the set $\{x \in V(T):v \in B_x\}$ induces a non-empty (connected) subtree of $T$, and 
				\item for each edge $uv \in E(G)$, there is some $x \in V(T)$ such that $u,v \in B_x$.
				
		\end{itemize}
		 
We refer to the sets $B_x$ in the decomposition as \textit{bags}. The \textit{width} of a decomposition is the maximum size of a bag minus 1.  The \textit{treewidth} of a graph $G$, denoted by $\tw(G)$, is the minimum width over all tree decompositions of $G$.  A \textit{path decomposition} of $G$ is a tree decomposition whose underlying tree is a path. The \textit{pathwidth} of a graph $G$, denoted by $\pw(G)$, is the minimum width over all path decompositions of $G$. For simplicity, we describe a path decomposition by $(B_1,B_2,\dots,B_n)$, where $B_i$ is the bag associated with the $i$-th vertex in the path. In such a decomposition, for each vertex $v$ of $G$, let $L(v)$ be the bag $B_i$ containing $v$ with $i$ minimum. If $L(v)=L(w)=B_i$ for distinct $v,w \in V(G)$, then replace $B_i$ by the two bags $B_i \setminus \{v\}$ and $B_i$. Now $L(v)\neq L(w)$. Repeat this step until $L(v) \neq L(w)$ for all distinct $v,w \in V(G)$. Such a path decomposition is said to be \textit{normalised}. Hence, every graph has a normalised path decomposition with width $\pw(G)$.

A graph $H$ is a \textit{minor} of a graph $G$ if $H$ is isomorphic to a graph formed from a subgraph of $G$ by contracting  edges.  When $H$ is a minor of $G$, for each vertex $v \in V(H)$ there is a connected subgraph $C$ of $G$ that contracts to form $v$ in the minor.  We call $C$ the \textit{branch set} of $v$. 

In a rooted forest $F$, the \textit{height} of a vertex $v$ in $F$ is the distance between $v$ and the root of the component of $F$ that contains $v$. The \textit{height} of $F$ is the maximum height over all vertices of $F$. The \textit{closure} of  $F$, denoted $\text{clos}(F)$, is the graph with vertex set $V(F)$ and edge set $\{xy: x \text{ is an ancestor of } y, x \neq y\}$.  The \textit{treedepth} of a graph $G$, denoted $\td(G)$, is the minimum height plus 1 of a forest $F$ such that $G \subseteq \text{clos}(F)$. Treedepth is equivalent to several other notions including minimal elimination tree height and is closely related to a number of graph invariants including pathwidth and treewidth; see \cite{bodlaender2,nesetril}.

\section{Proof of Theorem~\ref{thm:2-conn}}
\label{sec:2-conn}

\begin{lemma} Every 2-connected graph $G$ with circumference $t$ has treedepth at most $\floor{\tfrac{t}{2}}(t-1)+1$.
\label{lem:dfstree}
\end{lemma}

\begin{proof}  
Let $T$ be a \emph{depth-first} spanning subtree $T$ of $G$ rooted at some vertex $r$. Thus $G\subseteq \text{clos}(F)$. Say an edge $vw$ of $T$ has \emph{span} $|i-j|$, where $v$ and $w$ are respectively at height $i$ and $j$ in $T$. For each edge $vw$ of span $s\geq 2$, the $vw$-path in $T$ plus $vw$ forms a cycle of length $s+1$. Thus $s\leq t-1$. Consider a vertex $v$ in $G$. By Menger's Theorem, there are two internally disjoint $vr$-paths in $G$. Their union is a cycle of length at most $t$. Thus there is a $vr$-path $P$ in $G$ of length at most $\floor{\tfrac{t}{2}}$. Since each edge in $P$ has span at most $t-1$, the height of $v$ is at most $\floor{\tfrac{t}{2}}(t-1)$. Hence the height of $T$ is at most $\floor{\tfrac{t}{2}}(t-1)$. The result follows. 
\end{proof}

Theorem~\ref{thm:2-conn} follows directly from Lemma~\ref{lem:dfstree} since $\pw(G) \leq \td(G)-1$ (see \cite{nesetril}).

\section{Proof of Theorem~\ref{thm:disjointcycles}}
 \label{proofdisjointcycles}

A \emph{block} in a graph $G$ is a maximal 2-connected subgraph of $G$, or the subgraph of $G$ induced by a bridge edge or an isolated vertex. It is well known that the blocks of $G$ form a proper partition of $E(G)$. The \emph{block-cut-forest} $T$ of a  graph $G$ is defined as follows: $V(T)$ is the set of  cut-vertices and blocks of $G$, where a cut-vertex $v$ is adjacent to a block $B$ whenever $v\in B$. It is well known that $T$ is a forest, and if $G$ is connected, then $T$ is a tree called the \emph{block-cut-tree}.

\begin{lemma}  
Let $T$ be the block-cut-forest of a graph $G$.  Assume that $\pw(T) \leq n$ and $\pw(B) \leq m$ for each block $B$ of $G$.  Then $\pw(G) \leq (m+3)(n+1)-3$.  
\label{lem:bctree}
\end{lemma}  

\begin{proof} 
We proceed by induction on $\pw(T)$.  For the base case, say $\pw(T)=0$. Then $T$ has no edges, and each component of $G$ is 2-connected. Clearly, the pathwidth of $G$ equals the maximum pathwidth of the components of $G$. Thus $\pw(G)\leq m=(m+3)(n+1)-3$.

Now assume that $\pw(T)\geq1$. Since the pathwidth of $G$ equals the maximum pathwidth of the components of $G$, we may assume that $G$ is connected. Thus $T$ is connected. Let $(X_1,X_2,\dots,X_s)$ be a path decomposition of $T$ with width at most $n$. Choose vertices $x \in X_1$ and $y \in X_s$. Let $P$ be a maximal path in $T$ that contains an $xy$-path. Then $V(P) \cap X_i \neq \emptyset$ for all $i$. Let $X_i'=X_i - V(P)$; then $|X_i'| \leq |X_i|-1$. Now $(X_1',X_2',\dots,X_s')$ is a path decomposition of $T-V(P)$ with width at most $n-1$. By the maximality of $P$, the endpoints of $P$ are leaf vertices of $T$. No cut-vertex of $G$ is a leaf of $T$. Thus the endpoints of $P$ correspond to blocks. Say $P=b_1v_1b_2v_2\dots b_{s-1}v_{s-1}b_s$, where $b_i$ represents the block $B_i$ in $G$, and $v_i$ is a cut-vertex in $G$. For each $v_i$, let $C_{i,1}, C_{i,2},\dots,C_{i,t_i}$ be the blocks in $G$ corresponding to neighbors of $v_i$ in $T-V(P)$. Let $G_0:=\bigcup\{B_i:1\leq i\leq s\} \bigcup\{C_{i,j}:1 \leq i \leq s-1, 1 \leq j \leq t_i\}$. Let $G'$ be the union of the blocks not in $G_0$. Then $G=G_0 \bigcup G'$, as illustrated in Figure~\ref{fig:G_0andG}. 

\begin{figure}[htb]
	\centering \scalebox{.8}{\includegraphics{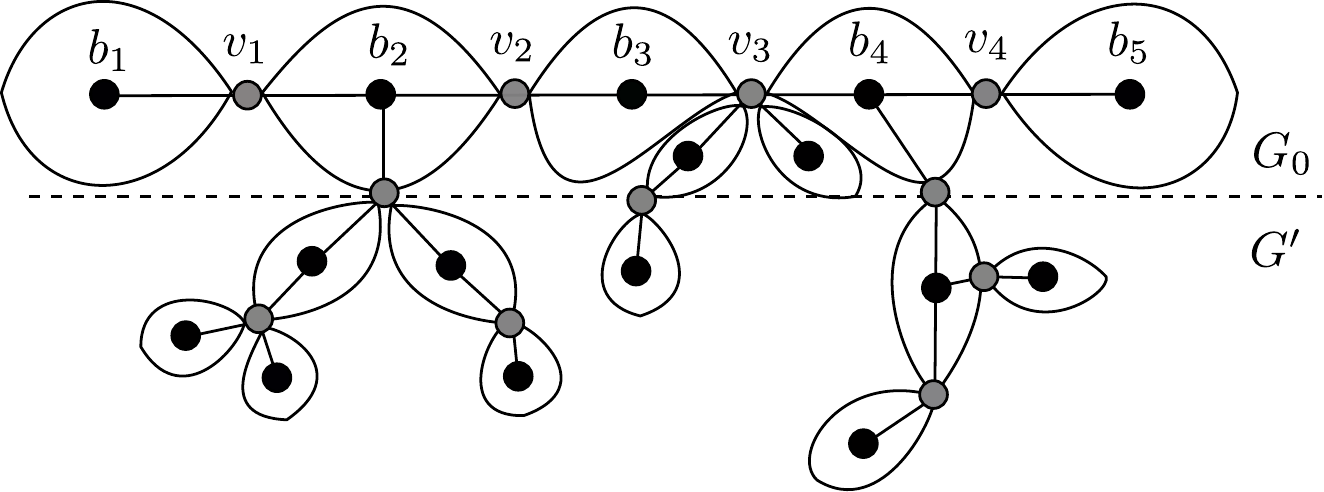}}
	\caption{\label{fig:G_0andG}Example of a block-cut tree: $P$ is the path $(b_1,v_1,b_2,v_2,b_3,v_3,b_4,v_4,b_5)$. The subgraphs $G_0$ and $G'$ respectively consist of the blocks above and below the dotted line. }
\end{figure} 

Let $T'$ be the forest obtained from $T-V(P)$ by removing the leaf vertices that correspond to cut-vertices in $G$.  This step removes all cut-vertices in $G$ that are not cut-vertices in $G'$, and the blocks that remain are blocks in $G'$.  Thus $T'$ is the block-cut-forest of $G'$.  Since $T'$ is a subgraph of $T-V(P)$, we have $\pw(T')\leq \pw(T-V(P)) \leq n-1$.  Furthermore, since each block $B$ of $G'$ is also a block of $G$, $\pw(B) \leq m$.  Let $G_1', G_2',\dots,G_r'$ be the components of $G'$.  By induction, $\pw(G_j') \leq (m+3)n-3$ for $1 \leq j \leq r$.  Let $(H_{j,1}, H_{j,2},\dots,H_{j,k_j})$ be a path decomposition of $G_j'$ with $|H_{j,\ell}| \leq (m+3)n-2$.  

We now construct a path decomposition of $G_0$.  For $1\leq i\leq s$, let $(X_{i,1},X_{i,2},\dots,X_{i,k_i})$ be a path decomposition of $B_i$ with $|X_{i,j}| \leq m+1$ for $1 \leq j \leq k_i$.  Define $X^+_{1,j}:=X_{1,j}\cup\{v_1\}$ and $X^+_{s,j}:=X_{s,j}\cup\{v_{s-1}\}$ and $X^+_{i,j}:=X_{i,j}\cup\{v_{i-1},v_i\}$ for $1< i<s$. For each $C_{i,j}$, let $(S_{i,j,1},S_{i,j,2},\dots,S_{i,j,k_{i,j}})$ be a path decomposition with $|S_{i,j,\ell}| \leq m+1$.  Define $S_{i,j,\ell}^+:=S_{i,j,\ell}\cup\{v_i\}$.  Denote by $T_{i,j}$ the sequence of bags $(S_{i,j,1}^+,S_{i,j,2}^+,\dots,S_{i,j,k_{i,j}}^+)$.  It is easily proved that 
$$(X^+_{1,1},\dots,X^+_{1,k_1},
T_{1,1},\dots,T_{1,t_1},
X^+_{2,1},\dots,X^+_{2,k_2},
T_{2,1},\dots,T_{2,t_2},
\dots,
T_{s-1,1},\dots,T_{s-1,t_{s-1}},
X^+_{s,1},\dots,X^+_{s,k_s})$$
is a path decomposition of $G_0$. The maximum bag size is at most $m+3$. 
Let $(Y_1,Y_2,\dots,Y_p)$ be a normalised path decomposition of $G_0$ with $|Y_i| \leq m+3$ for $1 \leq i \leq p$.  Then  
$L(v)\neq L(w)$ for $v \neq w$.

We now construct a path decomposition of $G$.  For each component $G_j'$ of $G'$, let $w_j$ be the cut-vertex in $G_0 \cap G_j'$.  Note that $w_j$ is distinct for each $G_j'$.  Replace the bag $L(w_j)$ with the bags
 
$$(L(w_j)\cup H_{j,1}, L(w_j) \cup H_{j,2}, \dots, L(w_j) \cup H_{j, k_j}).$$
  
The bag size is at most $m+3+(m+3)n-2=(m+3)(n+1)-2$. For simplicity, rename the decomposition $(Z_1,\dots,Z_q)$.  It remains to show that $(Z_1,\dots,Z_q)$ is a path decomposition of $G$.  For each edge $xy$ in $G$, we have $x,y \in Z_i$ for some $i$.  Suppose $v \in Z_i \cap Z_j$ for $j > i+1$.  Furthermore, assume $v \in V(G'-G_0)$ and without loss of generality, $v \in V(G_1')$.  Then by construction, $H_{1,r} \subset Z_i, H_{1,r+1} \subset Z_{i+1}, \dots , H_{1,r+j-i} \subset Z_j$ for some $r$.  $v \in H_{1,t}$ for $r \leq t \leq r+j-i$ so $v \in Z_s$ for $i \leq s \leq j$.  Now instead assume $v \in V(G_0)$.  Then by construction, $Y_r \subset Z_i$ and $Y_s \subset Z_j$ for some $r,s$ with $s \geq r$.  $v \in Y_t$ for $r \leq t \leq s$ so $v \in Z_{\ell}$ for $i \leq \ell \leq j$.      

We conclude that $(Z_1,\dots,Z_q)$ is a valid path decomposition. Since $|Z_i| \leq (m+3)(n+1)-2$, we have $\pw(G) \leq (m+3)(n+1)-3$.
\end{proof}

Let $T$ be a complete binary tree embedded in the plane as illustrated in Figure~\ref{fig:cbtree}.  Vertices at the same distance from the root are at the same level.  Number the leaf vertices from left to right; let $v_i$ be the leaf  labeled $i$ as shown.      

		\begin{figure}[htb]
		\centering \scalebox{.7}{\includegraphics{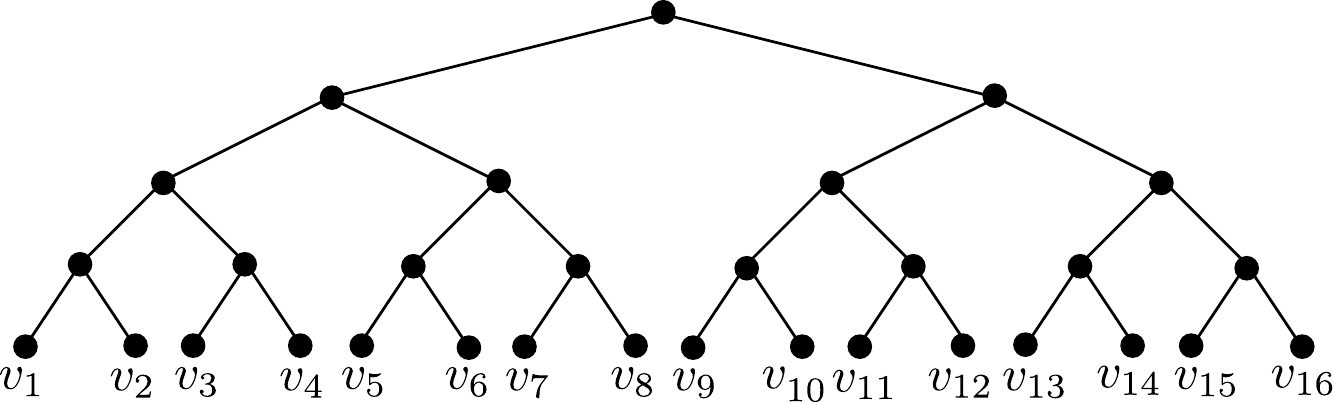}}
		\caption{\label{fig:cbtree}Left-to-right labeling of the leaves of the complete binary tree.}
		\end{figure}

		\begin{lemma}  Let $T$ be a complete binary tree with leaf vertices numbered as in Figure~\ref{fig:cbtree}.  Then the path in $T$ between $v_a$ and $v_b$ has length at least $2\log_2(b-a+1)$ where $a \leq b$.  
		\label{lem:leafpath}
		\end{lemma}
		
\begin{proof} Let $V_0$ be the set of all leaf vertices of $T$.  Let $V_i$ be the set of all vertices $u$ of $T$ such that the shortest path from $u$ to a vertex in $V_0$ has length $i$.  Since $T$ is a complete binary tree, each $u \in V_i$ has exactly $2^i$ descendants in $V_0$; furthermore, the descendants are $v_j, v_{j+1},\dots,v_{j+2^i-1}$ for some number $j$.  Consider the vertex $v_a$ and suppose $u \in V_i$ is an ancestor of $v_a$.  Then if $v_j \in V_0$ also has $u$ as an ancestor, then $j \in [a-(2^i-1),a+(2^i-1)]$.  

For all $b \geq a$, there exists $k$ such that $2^k \leq (b-a+1) < 2^{k+1}$.  Then, for $i <k$, $b \notin [a-(2^i-1),a+(2^i-1)]$, so $v_a$ and $v_b$ do not have a common ancestor in $V_i$.  However, $b \in [a-(2^j-1),a+(2^j-1)]$ for all $j \geq k$ and there exists some $j \geq k$ such that $v_a$ and $v_b$ have a common ancestor $u$ in $V_j$.  Then by the definition of $V_j$, the path $P_1$ from $v_a$ to $u$ has length $j$ and the path $P_2$ from $u$ to $v_b$ has length $j$.  Thus $P_1P_2$ is a path of length $2j$ from $v_a$ to $v_b$.  Since $2^k \leq b-a+1<2^{k+1}$ and $j \geq k$, we have $2j \geq 2\log_2(b-a+1)$.    
\end{proof}  

\begin{lemma}
Let $T$ be a forest with $\pw(T) \geq t\geq 1$.  Then $T$ contains a complete binary tree of height $t-1$ as a minor.  Moreover, for any vertex $v \in V(T)$, there is such a minor in $T$ with the property that $v$ is in the branch set of the root of the binary tree.    
\label{lem:cbtminor}
\end{lemma}
		
\begin{proof} 
Since the pathwidth of a graph equals the maximum pathwidth of its components, we may assume that $T$ is a tree. 
For a vertex $v$ of $T$, define the \textit{rooted pathwidth} of $T$ at $v$, denoted $\rpw(T,v)$, as the minimum width of a path decomposition of $T$ such that $v$ is in the last bag of the decomposition.  We say such a decomposition is \emph{rooted} at $v$.  

We prove, by induction on $|V(T)|$, that if $\rpw(T,v) \geq t$ for some vertex $v$ of a tree $T$, then $T$ contains   a complete binary tree of height $t-1$ as a minor with $v$ in the branch set of the root. Since $\rpw(T,v) \geq \pw(T)$, the result follows when $\pw(T) \geq t$.  

In the base case with $|V(T)|=2$, the rooted pathwidth at a given vertex is 1 and the tree trivially contains a complete binary tree of height 0 rooted at the given vertex.  

Now suppose $|V(T)| \geq 3$ and let $v$ be such that $\rpw(T,v) \geq t$. 
Let $w_1,w_2,\dots,w_d$ be the neighbors of $v$ and let $T_i$ be the component of $T-v$ rooted at $w_i$ for $1 \leq i \leq d$.  Let $r_i=\rpw(T_i,w_i)$. Without loss of generality,  $r_1 \geq r_2 \geq \dots \geq r_d$.  

Let $(X_{i,1},X_{i,2},\dots,X_{i,k_i})$ be a path decomposition of $T_i$ rooted at $w_i$ with width $r_i$.  For $2 \leq i \leq d$, let $X_{i,j}^+=X_{i,j} \cup \{v\}$.  Then 
$$(X_{1,1}, X_{1,2},\dots,X_{1,k_1},\{w_1,v\},X_{2,1}^+, X_{2,2}^+,\dots,X_{2,k_2}^+,X_{3,1}^+,X_{3,2}^+,\dots,X_{3,k_3}^+,\dots,X_{d,1}^+,X_{d,2}^+,\dots,X_{d,k_d}^+)$$
is a path decomposition of $T$ rooted at $v$ with  width $\text{max}\{r_1,r_2+1\}$.  Here we use the fact that $w_1\in X_{1,k_1}$. Thus $\rpw(T,v) \leq \text{max}\{r_1,r_2+1\}$. 

First suppose that $r_1 \geq r_2+1$.  Then $\rpw(T_1,w_1)=r_1\geq \rpw(T,v) \geq t$. By induction, $T_1$ contains a complete binary tree of height $t-1$ rooted at $w_1$ as a minor. Extend the branch set containing $w_1$ to include $v$. We  obtain a complete binary tree of height $t-1$ rooted at $v$ as a minor in $T$.  

Now suppose that $r_2+1 > r_1$. Then $r_1= r_2 \geq \rpw(T,v)-1 \geq t-1$.  By induction, $T_1$ and $T_2$ each contain a complete binary tree of height $t-2$ as a minor rooted at $w_1$ and $w_2$ respectively. Thus $T$ contains a complete binary tree  of height $t-1$ rooted at $v$ as a minor.
\end{proof}

To prove Theorem~\ref{thm:disjointcycles}, we need the following. Let $\mathcal{F}$ be a family of graphs.  For a graph $G$, a \textit{hitting set} $H$ of $\mathcal{F}$ is a set of vertices of $G$ such that $G-H$ contains no member of $\mathcal{F}$.  The family $\mathcal{F}$ is said to satisfy the \textit{Erd\H{o}s-P\'{o}sa property} if there is a function $f:\mathbb{N} \rightarrow \mathbb{N}$ such that for all graphs $G$, either $G$ contains $k$ vertex-disjoint members of $\mathcal{F}$ or  $G$ contains a hitting set $H$ of size at most $f(k)$.  Birmele, Bondy and Reed~\cite{bondy} proved that if $\mathcal{F}_t$ is the family of cycles of length at least $t$, then $\mathcal{F}_t$ satisfies the Erd\H{o}s-P\'{o}sa property with $f(k)=13t(k-1)(k-2)+(2t+3)(k-1)$. 
    
\begin{proof}[Proof of Theorem~\ref{thm:disjointcycles}] 
Since $G$ contains no $k$ vertex-disjoint cycles of length at least $t$, by the above-mentioned result of Birmele, Bondy and Reed~\cite{bondy}, there is  a hitting set $H \subseteq V(G)$ such that $|H| \leq h := 13t(k-1)(k-2)+(2t+3)(k-1)$.  Let $T$ be the block-cut-forest of $G-H$. Define 
\begin{align*}
i:=\lfloor\log_2(k-1)(2h-2k+1)\rfloor+1\quad\text{and}\quad
j:=\lceil \tfrac{t}{2}+\log_2(h-k+1)\rceil.
\end{align*}
Since $G$ is $(k+1)$-connected, $|H| \geq k$.  Hence $h \geq |H| \geq k$, and $i$ and $j$ are well-defined.  

First suppose that $\pw(T) \leq i+j$. Since $H$ is a hitting set, $G-H$ has circumference at most $t-1$. Thus the 2-connected blocks of $G-H$ have pathwidth at most $\floor{\tfrac{t-1}{2}}(t-2)$ by Theorem~\ref{thm:2-conn}. The blocks that are not 2-connected consist of bridges or isolated vertices, which have pathwidth at most 1. By Lemma~\ref{lem:bctree} with $m=\floor{\tfrac{t-1}{2}}(t-2)$ and $n=i+j$, we have $\pw(G-H) \leq (\floor{\tfrac{t-1}{2}}(t-2)+3)(i+j+1)-3$.  Add  $H$ to each bag of an optimal path decomposition of $G-H$ to obtain a path decomposition of $G$ with width at most $(\floor{\tfrac{t-1}{2}}(t-2)+3)(i+j+1)-3 + h\in \mathcal{O}(t^3+tk^2)$.

It remains to handle the case when $\pw(T)>i+j$. We claim, however, that this case does not occur. Suppose it does and assume $\pw(T)>i+j$. By Lemma~\ref{lem:cbtminor}, $T$ contains a complete binary tree $T'$ of height $i+j$ as a minor. It is well known and easily proved that if a graph $A$ contains a graph $B$ with maximum degree 3 as a minor, then $A$ contains a subdivision of $B$ as a subgraph. Thus, $T$ contains a subdivision $S$ of $T'$ as a subgraph. By taking $S$ maximal, each leaf of $S$ is a leaf  of $T$. 
 
For each $v \in H$, let $d(v)$ be the number of leaves $u$ of $S$ such that $v$ is adjacent in $G$ to some vertex in the block corresponding to $u$ (in which case we say that $v$ is \emph{adjacent} to $u$). Since $G$ is $(k+1)$-connected, each leaf  of $S$ has at least $k$ neighbors in $H$. Since $S$ contains $2^{i+j}$ leaves, $\sum_{v \in H} d(v) \geq k\,2^{i+j}$. Let $H=\{v_1,v_2,\dots,v_h\}$ and $d_m:=d(v_m)$. Without loss of generality,  $d_1 \geq d_2 \geq \dots \geq d_h$. Since $d_m \leq 2^{i+j}$ for $1 \leq m \leq h$,  $$d_k+d_{k+1}+\dots+d_h \geq k\,2^{i+j}-(d_1+d_2+\dots+d_{k-1}) \geq k\,2^{i+j}-(k-1)2^{i+j}=2^{i+j}.$$ Hence $d_1 \geq d_2 \geq \dots \geq d_k \geq 2^{i+j}/(h-k+1)$. Let $X:=\{v_1,v_2,\dots,v_k\}$.

Since $T'$ has height $i+j$, there are $2^i$ pairwise disjoint subtrees $T_1,T_2,\dots,T_{2^i}$ in $S$, each a subdivision of a complete binary tree of height $j$, such that for $1\leq m\leq 2^i$, the leaves of $T_m$ are leaves of $S$ and the root of  $T_m$ is at height $i$ in $T'$, as illustrated in Figure~\ref{fig:CBT}. For each $v \in X$, we say the pair $(v,T_m)$ is \textit{good} if $v$ is adjacent to at least $2^{j-1}/(h-k+1)$ leaves of $T_m$. We claim that each $v \in X$ is in at least $k$ good pairs. Suppose for the sake of contradiction that some $v \in X$ is in at most $k-1$ good pairs. Then 
$$\frac{2^{i+j}}{h-k+1} \leq d(v) \leq (k-1)2^j+(2^i-k+1)\frac{2^j}{2(h-k+1)}. $$ Thus $2^i \leq (k-1)(2h-2k+1)$, which contradicts the definition of $i$. Thus each $v \in X$ is in at least $k$ good pairs.  Since $|X|=k$, there is a distinct $T_m$ for each $v \in X$ such that $(v,T_m)$ is a good pair. 

\begin{figure}[htb]
\centering \scalebox{.7}{\includegraphics{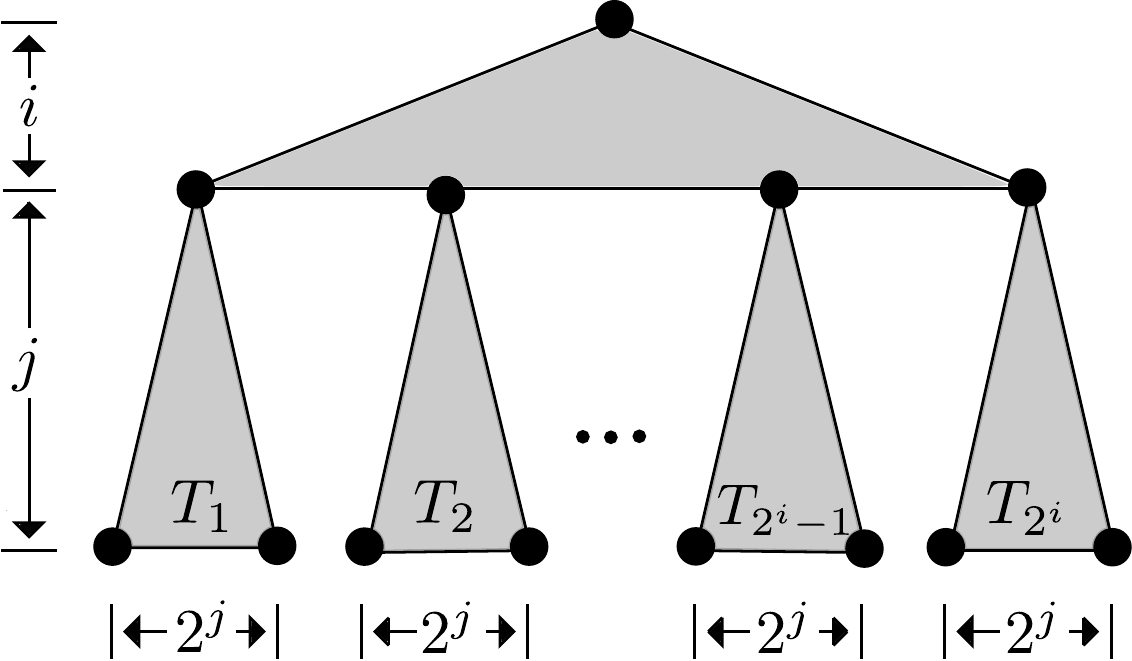}}
\caption{\label{fig:CBT}Complete binary tree of height $i+j$ with $2^i$ disjoint subtrees of height $j$.}
\end{figure}

For each such pair $(v,T_m)$, label the leaf vertices of $T_m$ as in Figure~\ref{fig:cbtree}. Since $v$ is adjacent to at least $2^{j-1}/(h-k+1)$ leaf vertices, there are two leaves $x$ and $y$ labeled $a$ and $b$ respectively such that $b-a+1 \geq \frac{2^{j-1}}{h-k+1}$. Then by Lemma~\ref{lem:leafpath}, there is a path $P$ of length at least $2\log_2\left(\frac{2^{j-1}}{h-k+1}\right) = 2j-2-2\log_2(h-k+1)$ in $T_m$ between $x$ and $y$. Thus $vPv$ is a cycle of length  $2j-2\log_2(h-k+1)\geq t$ in $T_m\cup\{v\}$. Since the $T_m$ are pairwise disjoint, we have $k$ pairwise disjoint cycles $C_1,C_2,\dots,C_k$ of length at least $t$ in $T \cup H$. 

We now construct pairwise disjoint cycles  $C_1',C_2',\dots,C_k'$ in $G$. Say $C_1=v_1B_1v_2B_2 \dots B_rv_1$, where $v_1 \in H$, $v_i$ is a cut-vertex in $G-H$ for $2 \leq i \leq r-1$, and $B_i$ is a block in $G-H$. The vertex $v_1$ is adjacent to a vertex $x$ in $B_1$. Let $P_1$ be a path from $x$ to $v_2$ in $B_1$. Next, for $2 \leq i \leq r-1$, let $P_i$ be a path from $v_i$ to $v_{i+1}$ in $B_i$, such that if there is a vertex $v$ in $B_i\cap V(C_j)$ for some $j \neq 1$, then choose $P_i$ such that $v \notin V(P_i)$, as illustrated in Figure~\ref{fig:newcycles}. Since each vertex in $S$ has degree at most $3$, there is at most one such vertex $v$ to be avoided. Therefore, since $B_i$ is $2$-connected, such a $P_i$ exists. For $B_r$, let $P_r$ be a path from $v_r$ to $y$ in $B_r$, where $y$ is a neighbor of $v_1$. Let $C_1'=v_1xP_1v_2P_2v_3\dots P_ryv_1$.  From each $C_i$, construct $C_i'$ in $G$ in this same manner. The cycles $C_1',C_2',\dots,C_k'$ by construction are pairwise disjoint with length at least $t$ in $G$, which is a contradiction.     
\end{proof}  

\begin{figure}[htb]
\centering \scalebox{.7}{\includegraphics{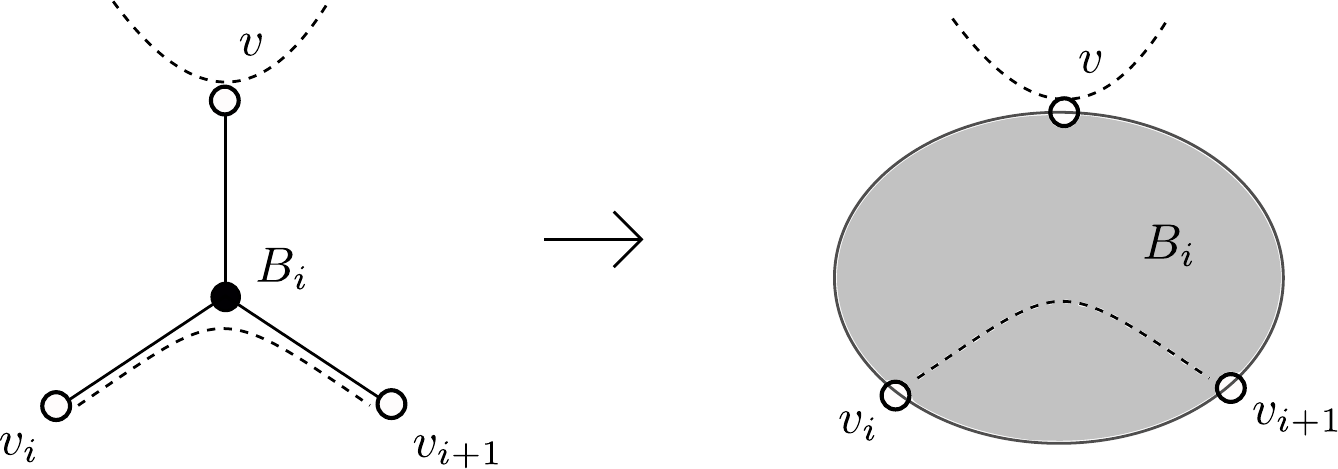}}
\caption{\label{fig:newcycles}Routing the cycles in $G$.}
\end{figure}

\section{Relationship to Forbidden Minors}
\label{minors}

Another way to describe a graph $G$ with circumference $t-1$ is to say $G$ is $C_t$-minor-free where $C_t$ is a cycle on $t$ vertices.  Our two main theorems can thus be restated in terms of minors:

		\begin{theorem}  Let $G$ be a 2-connected $C_t$-minor-free graph.  Then $\pw(G) \leq \floor{\frac{t-1}{2}}(t-2)$.
		\label{thm:new2-conn}
		\end{theorem} 
		 
Let $C_{t,k}$ be the graph consisting of $k$ disjoint cycles of length $t$.

		\begin{theorem}  Let $G$ be a $(k+1)$-connected $C_{t,k}$-minor-free graph.  Then $\pw(G) \leq \mathcal{O}(t^3+tk^2)$. 
		\label{thm:newdisjointcycles}
		\end{theorem}

These results suggest the following definition. For a graph $H$, let $g(H)$ be the minimum integer for which there exists a number $c=c(H)$ such that every $g(H)$-connected $H$-minor-free graph has pathwidth at most $c$. Mader~\cite{mader} exhibited a function $\ell$ such that every $\ell(H)$-connected graph contains $H$ as a minor. 
 (Kostochka~\cite{kost1,kost2} and Thomason~\cite{thomason} independently proved that if $t=|V(H)|$ then $\ell(H)\leq\ell(K_t)\in\Theta(t\sqrt{\log t})$.)\ Thus every $H$-minor-free $\ell(H)$-connected graph has bounded pathwidth (since there is no such graph). Hence $g(H)$ is well-defined, and  $g(H)\leq \ell(H)$. We conclude with some observations about $g(H)$. 
  
For some graphs, $g(H)=\ell(H)$. For example, $g(K_5)=\ell(K_5)=6$ (since every 6-connected graph contains $K_5$ as a minor, but 5-connected planar (and thus $K_5$-minor-free) graphs have unbounded pathwidth). 
 
On the other hand,  $g(H)$ and $\ell(H)$ can be far apart. For example, we showed that $g(C_t)=2$ but $\ell(C_t)\geq t-1$ since $K_{t-1}$ is $(t-2)$-connected and contains no $C_t$-minor. 

Observe that if $H_1$ is a minor of $H_2$, then $g(H_1)\leq g(H_2)$. Thus, for each integer $c$, the class $\mathcal{H}_c:=\{H:g(H)\leq c\}$ is minor-closed. By Robertson and Seymour's graph minor theorem, for each $c$, there is a finite set of minimal excluded minors for $\mathcal{H}_c$.

Bienstock, Robertson, Seymour and Thomas~\cite{bienstock} proved that for every forest $F$, every graph with pathwidth at least $|F|-1$ contains $F$ as a minor. Thus $g(F)=0$. Moreover, since complete binary trees have unbounded pathwidth, $g(F)=0$ if and only if $F$ is a forest. And $K_3$ is the only minimal excluded minor for $\mathcal{H}_0$.

There is no graph $H$ with $g(H)=1$ since the pathwidth of a graph equals the maximum pathwidth of its connected components. 

We showed that $g(C_t)=2$ for all $t\geq3$. It is an interesting open problem to characterise the graphs $H$ with $g(H)=2$. (An answer is conjectured below.)\ 

The following example is important. Consider $G_0:=K_3$ embedded in the plane. For $i\geq 0$, construct $G_{i+1}$ from $G_i$ as follows: for each edge $vw$ on the outerface of $G_i$, add one new vertex adjacent to $v$ and $w$. Thus $G_i$ is 2-connected and outerplanar. Hence $G_i$ is $K_4$-minor-free and $K_{2,3}$-minor-free.  Observe that the dual of $G_i$ contains a complete binary tree of height $i$ as a minor, which has pathwidth $i$. By a result of Bodlaender and Fomin~\cite{bodlaender}, the class $\{G_i:i\geq 0\}$ has unbounded pathwidth. Hence $g(K_4)\geq3$ and $g(K_{2,3})\geq 3$. 

Dirac~\cite{dirac} proved that every 3-connected graph has a $K_4$-minor. Thus $g(K_4)=\ell(K_4)=3$. 

An unfinished result of Ding~\cite{ding} implies that, for some function $f$, every 3-connected $K_{2,t}$-minor-free graph has pathwidth at most $f(t)$, implying $g(K_{2,t}) \leq 3$. Thus $g(K_{2,t}) \geq  g(K_{2,3})$ and $g(K_{2,t})=3$ for $t\geq 3$  (assuming Ding's result).

We proved that $g(C_{t,k})=k+1$ for all $t\geq 3$, where the lower bound follows from the example given after the statement of Theorem~\ref{thm:disjointcycles}. This leads to the following lower bound on $g(H)$: If $H$ contains $k$ disjoint cycles, then $C_{3,k}$ is a minor of $H$, and $g(H)\geq k+1$. This observation can be strengthened as follows.  A \emph{transversal} in a graph $H$ is a set $X$ of vertices such that $H-X$ is acyclic. Let $\tau(H)$ be the minimum size of a transversal in $H$. Note that if $H$ is a minor of $G$, then $\tau(H)\leq \tau(G)$. 

\begin{proposition}
\label{prop}
$g(H)\geq \tau(H)+1$ for every graph $H$ with $\tau(H)\geq1$. 
\end{proposition}

\begin{proof} 
Suppose on the contrary that $g(H)\leq \tau(H)$ for some graph $H$. Let $G$ be the graph obtained from the complete binary tree of height $h$ by adding $\tau(H)-1$ dominant vertices. Then $G$ is $\tau(H)$-connected, and $\tau(G)=\tau(H)-1$, implying $G$ is $H$-minor-free. By the definition of $g(H)$, for some $c=c(H)$, the pathwidth of $G$ is at most $c$. This is a contradiction for $h>2c$, since $G$ has pathwidth $\ceil{\frac{h}{2}}+\tau(H)-1$. Therefore $g(H)\geq \tau(H)+1$. 
\end{proof}

We have described three minor-minimal graphs $H$ with $g(H)=3$. Namely, $K_4$, $K_{2,3}$ and $K_3\cup K_3$. (It is easily seen that these graphs are minor-minimal.)\ There is one more key example. Let $Q$ be the octahedron graph $K_{2,2,2}$ minus the edges of a triangle. Observe that $\tau(Q)=2$, and thus $g(Q)\geq 3$ by Proposition~\ref{prop}. Moreover, $Q$ contains no $K_4$, $K_{2,3}$ or $K_3\cup K_3$  minor. 

\begin{conjecture}
\label{conj}
The minimal excluded minors for $\mathcal{H}_2$ are $\{K_4, K_{2,3}, K_3\cup K_3, Q\}$. 
\end{conjecture}

It is well known that $H$ is outerplanar if and only if $H$ contains no $K_4$ or $K_{2,3}$ minor, and it follows from a result of Lov{\'a}sz~\cite{Lovasz65} that $\tau(H)\leq 1$ if and only if $H$ contains  no $K_4$, $K_3\cup K_3$ or $Q$ minor. Thus Conjecture~\ref{conj} is equivalent to saying that $g(H)\leq 2$ if and only if $H$ is outerplanar and $\tau(H)\leq 1$.

In the above  examples $H$ is planar. Planarity is significant for these types of questions since the class of $H$-minor-free graphs has bounded treewidth if and only if $H$ is planar~\cite{robertson}. However, $g(H)$ is well-defined for all graphs, and is interesting for certain non-planar graphs. For example,  B\"ohme~\emph{et~al.}~\cite{bohme08} proved that there is a function $n$ such that every 7-connected graph with at least $n(k)$ vertices contains $K_{3,k}$ as a minor. That is, every 7-connected $K_{3,k}$-minor-free graph has less than $n(k)$ vertices, implying $g(K_{3,k})\leq 7$. More generally, B\"ohme~\emph{et~al.}~\cite{bohme02} conjectured that for all $a,k$ there is an integer $n(a,k)$ such that every $(2a+1)$-connected graph on at least $n(a,k)$ vertices contains $K_{a,k}$ as a minor. This would imply that $g(K_{a,k})\leq 2a+1$.

In general, it would be interesting if some function of $\tau(H)$ was an upper bound on $g(H)$. Or is there a family of graphs $H$ with bounded transversals, but with $g(H)$ unbounded?

\section*{Notes Added in Proof}

Fiorini and Herinckx \cite{fiorini} recently improved the above-mentioned result of Birmele, Bondy and Reed~\cite{bondy} 
by showing that cycles of length at least $t$ satisfy the Erd\H{o}s-P\'{o}sa property with $f(k)=\mathcal{O}(tk\log k)$ (which is optimal for fixed $k$ or fixed $t$). It follows that the $\mathcal{O}( t^3+tk^2)$ bound in Theorem~\ref{thm:disjointcycles} can be improved to  $\mathcal{O}(t^3+tk\log k)$. 

In an early version of this paper, the graph $Q$ was omitted from Conjecture~\ref{conj}. Proposition~\ref{prop} and the importance of  $Q$ were jointly observed with J\'anos Bar\'at and Gwena\"el Joret. 

Gwena\"el Joret also pointed out the following alternative proof of a slightly weaker version of Theorem~\ref{thm:2-conn}. Let $G$ be a 2-connected graph with circumference $t$. Let $p$ be the number of edges in the longest path in $G$. Dirac~\cite{diraca} proved that $t>\sqrt{2p}$. Thus $p<\ceil{\frac{t^2}{2}}$. That is, $G$ contains no path on $\ceil{\frac{t^2}{2}}$ edges. Hence $G$ contains no path on $\ceil{\frac{t^2}{2}}$ edges as a minor. Bienstock~\emph{et~al.}~\cite{bienstock} proved that every graph that excludes a fixed forest on $k$ edges as a minor has pathwidth at most $k-1$. Thus $G$ has pathwidth at most $\ceil{\frac{t^2}{2}}-1$.  

Thanks J\'anos and Gwen.

\end{document}